\documentclass[letterpaper, 10 pt, conference]{ieeeconf} 
\usepackage{amsmath}									% Math
\usepackage{amssymb}
\usepackage{cite}
\usepackage[usenames,dvipsnames]{color}
\usepackage{pgfplots}
 \pgfplotsset{compat=newest} 
  \pgfplotsset{plot coordinates/math parser=false}
  
\usepackage{pstool} 
\usepackage{psfrag}

 \usepackage{graphicx}
\usepackage{epstopdf} 
\DeclareGraphicsExtensions{.pdf,.png,.jpg,.eps,.ps}

\newcommand{\mset}{\mathbf{M}}
\newcommand{\graph}{\mathbf{G}}
\newcommand{\reels}{\mathbb{R}}

\newcommand{\tran}{\text{Tr}\,}

\newtheorem{theorem}{ Theorem}[section]
\newtheorem{definition}{ Definition}
\newtheorem{lemma}[theorem]{ Lemma}
\newtheorem{proposition}[theorem]{ Proposition}
\newtheorem{corollary}[theorem]{  Corollary}

\usepackage{comment}

\newif\ifproofs
\proofstrue
\IEEEoverridecommandlockouts  
\author{Matthew Philippe
\thanks{Matthew Philippe is a F.N.R.S./F.R.I.A. fellow, {\tt\small matthew.philippe@uclouvain.be}.}
\and 
Rapha\"{e}l M. Jungers 
\thanks{Rapha\"{e}l M. Jungers  is a F.R.S./F.N.R.S. research associate {\tt\small raphael.jungers@uclouvain.be}.}
\thanks{ Both authors are  with the ICTEAM Institute at the Universit\'e
 catholique de Louvain, Avenue Georges Lemaitre 4, B-1348 Louvain-la-Neuve, Belgium.
  The scientific responsibility rests with its authors.} 
  \thanks{Research supported by the Belgian Interuniversity Attraction Poles, and by the ARC grant 13/18-054 from the
Communaut\'{e} fran\c{c}aise de Belgique. }
}
\title{
Converse Lyapunov  theorems for discrete-time linear \\ switching systems with regular switching sequences.
 \\ 
		\vspace*{10pt}}
\begin{document}

\maketitle

 \begin{abstract}
  We present a stability analysis framework for the general 
  class of discrete-time linear switching systems
  for which the switching sequences belong to a regular language.
   They admit arbitrary switching systems as special cases.

Using recent results of X. Dai on the asymptotic growth rate of such systems,
we introduce the concept of \emph{multinorm} as an algebraic tool for stability analysis.

We conjugate this tool with two families of multiple quadratic Lyapunov functions,  parameterized by an integer $T \geq 1$, 
and obtain converse Lyapunov Theorems for each.

  Lyapunov functions of the first family associate one quadratic form per state of
  the automaton defining the switching sequences. They are made to decrease after every $T$ successive time steps. 
  The second family is made of the \textit{path-dependent} Lyapunov functions 
	of Lee and Dullerud. They are parameterized by an amount of memory $(T-1) \geq 0$. 

Our converse Lyapunov theorems are finite.  More precisely,
we give sufficient conditions on the asymptotic growth rate of a stable system
under which one can compute an integer parameter $T \geq 1$ for which both types of Lyapunov functions exist.
                 As a corollary of our results, we formulate an 
   arbitrary accurate approximation scheme for estimating the asymptotic growth rate
   of  switching systems with constrained switching sequences.
    
 \end{abstract}

\section{Introduction}
\label{sec:intro_ref}

%+++
A  switching system is a dynamical system whose state dynamics vary between different operating modes according to a switching
sequence. 
These modes can be represented by a collection of differential or difference equations. 
A discrete-time linear switching system is a switching system where each mode is represented by a linear difference equation. It
takes the form
\begin{align*}
&x_{t+1} = A_{\sigma(t)}^{} x_t,
\end{align*}
with $x_0 \in \reels^n$, $A_{\sigma(t)} \in \mset \subset \reels^{n \times n}$, 
where  $\mset$ is a finite set of $N$ matrices and the \textit{mode} at time $t$, $\sigma(t) \in \{1, \ldots, N\}$, 
represents a particular matrix in $\mset$. The sequence of modes $\sigma(0), \sigma(1), \ldots$, is the \textit{switching sequence} of the
system.

 % stability
There has been a lot of attention on the subject in recent years. 
Applications can be found in many domains, from the study of viral mutations \cite{HeMiOAMS} to congestion control 
in computer networks \cite{ShWiAPSM}, networked control systems \cite{JuDIFSOD}, 
and in the broad field of theoretical computer sciences \cite{JuTJSR}.  More generally, classical methods for simulating more complex systems (such as Hybrid or Cyber-Physical Systems) often boil down to the analysis of such a switching system.

 % stability analysis tools
Given the vast range of applications for discrete-time linear switching systems, 
questions relating to their stability have received a great interest. 
Surveys on the subject can be found in  \cite{ShWiSCFS, LiAnSASO, LiMoBPIS}.

 % unconstrained case
Many studies have considered arbitrary switching between modes (e.g.,
\cite{AgLiLASC, AhJuPaRoJSRP, AnShSC, JuTJSR}), 
in which $\sigma(t)$ can take any value in $\{1, \ldots, N\}$ at any time. 
In that scenario, a quantity known as the \textit{joint spectral radius} (JSR), 
introduced by Rota and Strang \cite{RoStANOT}, 
has been shown to be of particular interest as it corresponds to 
the worth case growth rate of the system. 
Because of its importance, there has been a large effort to develop approximation algorithms for this quantity 
(see for example \cite{AhJuPaRoJSRP, JuTJSR, BlNeOTAO} and references therein.).

 % constrained case

Other studies have considered systems whose switching sequences $\sigma(0), \sigma(1), \ldots$ are constrained to be in some  
\textit{regular language}. Regular languages are languages generated by a finite automaton.
A finite automaton is a (strongly) connected, directed, 
labeled graph $\graph(V,E)$,
on a set of nodes $V$ and edges $E$. 
Each edge $e = (v_i, v_j, \sigma)$ between two nodes $v_i$ and $v_j$ carries a \emph{label} $\sigma \in \{1, \ldots, N\}$, where $N$ is the number of different labels.
A given sequence of labels $\sigma(0)$, $\sigma(1)$, $\ldots$ 
is a \textit{word} in the language generated by the automaton if there exists a path in the graph such that, 
$\forall \, t \geq 0$, the label on the $(t+1)$'th edge in the path corresponds to $\sigma(t)$.

 We refer to the system on the set of $N$ matrices $\mset$ and graph $\graph$ on $N$ labels as $S(\graph,\mset)$. 
 Since these systems have their switching signals \textit{constrained} by an automaton, 
 we naturally call them \textit{constrained switching systems}.

 Constrained switching systems admit arbitrary switching systems as special case.
  They can be found in works including those of Bliman and Ferrari-Trecate \cite{BlFeSAOD},
   Lee and Dullerud \cite{LeDuUSOD, LeDuUSSO};  Lee and Khargonekar \cite{LeKhDASO, LeKhOORF};
     Essick et al. \cite{EsDuCOLS}, and many others.
 With the exception of \cite{BlFeSAOD}
  these works tackle several control problems for switching systems 
  (stabilizing controller design, disturbance attenuation, switching sequence design, etc.), 
  showing the importance of these systems in engineering.
 These approaches rely on sufficient
  conditions for stability established
   by the existence of \textit{multiple quadratic  Lyapunov functions}.

 Multiple Lyapunov functions have proven to be effective stability analysis tools
  for switching and hybrid systems \cite{BrMLFA, LiAnSASO}. 
 Multiple \textit{quadratic} Lyapunov functions, however, usually only  provide \textit{sufficient}
  conditions for stability \cite{DaBePDLF, DaRiSAAC}. 
 In some cases, they can be embedded within a \textit{hierarchy} of
  multiple quadratic Lyapunov functions \cite{BlFeSAOD, LeDuUSOD}, 
  to \textit{asymptotically} provide a necessary condition for stability as well.
 A notable example of such a hierarchy is given by the path-dependent Lyapunov functions of Lee and Dullerud \cite{LeDuUSOD}.\\
 
 One of the main motivations behind this paper is to provide a \textit{finite} version
  of these asymptotically necessary conditions for the general class of constrained switching systems 
$S(\graph,\mset)$.

Our main contribution takes the form of a \textit{Converse Lyapunov Theorem}.
For any $\epsilon > 0$, if the worst-case exponential growth
rate of a stable system is smaller than $1 - \epsilon$, one can compute
an integer $T(\epsilon)$ such that $S(\graph, \mset)$ admits a multiple
quadratic Lyapunov function, with a quadratic form per node of $\graph$, decreasing after every $T(\epsilon)$
time steps.

We then analyze \textit{path-dependent Lyapunov functions} within our framework.
We show if a system has a multiple quadratic Lyapunov function decreasing after $T$ times steps as described above,
then it also admits a path-dependent Lyapunov function with memory $(T-1)$.
As a corollary, our converse Lyapunov theorems are also valid for path-dependent Lyapunov functions.
These results translate into an arbitrary accurate approximation scheme for estimating
the exponential growth rate of $S(\graph, \mset).$ \\

 The plan of the paper is as follows.
 
  {Section \ref{sec:CJSR}} introduces the reader to the notion of Constrained Joint Spectral radius (CJSR) describing the stability of a system.

	In {Section \ref{sec:multi}}, 
	we present a necessary and sufficient condition for stability based on multiple
	 Lyapunov functions composed of a different norm per node in $\graph$. We call these sets of norms Lyapunov \textit{multinorms}.
	
	{Section \ref{sec:converse}} is split in three subsections.
	 Subsection \ref{subsec:quad1} presents sufficient conditions on the CJSR of a given system under which it admits a \textit{quadratic Lyapunov multinorm}. 
	 In Subsection \ref{subsec:quad2}, we present the converse Lyapunov theorem and approximation scheme discussed above.
	 In Subsection \ref{subsec:pathdep} we study  path-dependent Lyapunov  functions \cite{LeDuUSOD}  within our framework. 
	 
	 Before concluding the paper, Section \ref{sec:examp} proposes an example with a numerical illustration.

\section*{Notations and conventions}
\label{sec:prel_note}
%%%%%%%%%%%%%%%%%%%%%%%%%%%%%%%%%%%%%%%%%%%%%%%%%%%%%%%%%%%%%%%%%%%%%%%%%%%%%

By $|{\cdot}|$ we denote a vector norm and by $\|{\cdot}\|$ an induced matrix norm. 

Given a matrix $A \in \reels^{n \times m}$, its transpose is denoted $A^\tran \in \reels^{m \times n}$. For $Q \in \reels^{n \times n} $ 
symmetric, we write $Q \succ 0$ if it is positive definite, i.e. $\forall x \in \reels^n$, $x \neq 0$, $x^\tran Q x > 0$.

For any directed labeled graph $\graph(V,E)$, the number of nodes in the graph is denoted by $|V|$.
 We use $p \in \graph$ to represent the fact that $p$ is a path in $\graph$,
 that is $p$ is a sequence of connected consecutive edges.
  The length of a path $p \in G$ is the number of edges it contains, and is written $|p|$.

 Given a system $S(\graph,\mset)$ and a path $p \in \graph$ of length t carrying a sequence of successive labels $\sigma(0),  \cdots \sigma(t-1)$, 
 the matrix $A_p$ is the \textit{associated matrix product},	
  $$A_p = A_{\sigma(t-1)} \cdots A_{\sigma(0)},$$
   with $A_{\sigma(i)} \in \mset$, $0 \leq i \leq t-1$. 
     Given any constant $\alpha \in \reels$, we denote by $\alpha \mset$ the set of scaled matrices $\{\alpha A, \, A\in \mset \}$.
   For any integer $T \geq 1$ and set of matrices $\mset$,   the set $\mset^T$ contains all products of $T$ matrices in $\mset$. 
   
%%%%%%%%%%%%%%%%%%%%%%%%%%%%%%%%%%%%%%%%%%%%%%%%%%%%%%%%%%%%%%%%%%%%%%%%%%%%%%%%%%%
\section{Constrained Joint Spectral Radius}
\label{sec:CJSR}
%%%%%%%%%%%%%%%%%%%%%%%%%%%%%%%%%%%%%%%%%%%%%%%%%%%%%%%%%%%%%%%%%%%%%%%%%%%%%%%%%%%

Given a graph $\graph$ representing a finite automaton on $N$ labels and 
a set of $N$ matrices $ \mset \subset \reels^{n \times n}$, 
the system $S(\graph, \mset)$ is said to be \emph{stable} if for any admissible switching sequence and any $x_0 \in \reels^n$, 
$$ \lim_{t \rightarrow \infty} A_{\sigma(t)}\cdot \ldots \cdot A_{\sigma(0)}x_0  = 0.$$

Switching systems with arbitrary switching are a special case of constrained switching systems. A graph $\graph$ with one node and one self-loop for each mode allows for arbitrary switching (Figure \ref{fig:g-star}). More generally, any path-complete graph \cite{AhJuPaRoJSRP} generates arbitrary switching sequences.

\begin{figure}[ht]
\centering
\includegraphics[scale=0.7]{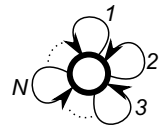}
\caption{Automaton generating arbitrary switching sequences.}
\label{fig:g-star}
\end{figure}

For arbitrary switching systems, given a finite 
set of matrices $ \mset \subset \reels^{n \times n}$, 
the quantity 
\begin{align*}
& \hat \rho(\mset) = \lim_{t \rightarrow \infty} \max \left \{ \| A_{\sigma(t-1)}^{\,} \cdot \ldots \cdot A_{\sigma(0)}^{\,} \|^{1/t}:  \right . \\
& \left . \qquad \qquad \qquad \qquad A_{\sigma(i)}^{\,}  \in \mset, \, 0 \leq i \leq t-1  \right \}
\end{align*}
is known as the \textit{joint spectral radius} (JSR) of the system and defines the asymptotic rate of growth of the arbitrary switching system.

To the best of our knowledge, 
 the \textit{constrained joint spectral radius} (CJSR) was first defined by Dai 
 \cite{DaAGTS}\footnote{Under the name "joint spectral radius with constraints".}
  as a generalization of the joint spectral radius.

\begin{definition}
The \emph{constrained joint spectral radius} $\hat \rho( \graph, \mset)$ of a system $S( \graph, \mset)$ is defined as 
\begin{equation}
\hat \rho(\graph, \mset) = \lim_{t \rightarrow \infty} \hat{\rho}_t(\graph, \mset),
\end{equation}
where $\hat{\rho}_t(\graph, \mset) = \max \left \{  \|A_p\|^{1/t} : p \in \graph; \, |p| = t \right \} $ is the maximum growth rate up to time $t$.
\label{def:GJSR}
\end{definition}

Remark that, by its definition, the CJSR is positive, 
homogeneous in the set of matrices $\mset$, i.e., for $\alpha \in \reels$
$$\hat \rho(\graph, \alpha \mset) = | \alpha|  \hat \rho(\graph, \mset),$$
and is independent of the norm $\|\cdot\|$
 used (by equivalence of norms in $\reels^n$).

The JSR of a switching system on a set $\mset$ of matrices is always an upper bound for the CJSR of any constrained system on this set. 

%% To move away.
More generally : 
\begin{proposition}
For any finite set $\mset$ of matrices, given any two graphs $\graph_1$ and $\graph_2$ 
such that for any path $p_1 \in \graph_1$, 
there exists a path $p_2 \in \graph_2$ carrying the exact same labels, then
$$\hat \rho(\graph_1, \mset) \leq \hat \rho(\graph_2, \mset).\\$$
\end{proposition}

The following result links the CJSR of a system with its stability.  The statement is reformulated from \cite{DaAGTS}, Corollary 2.8.

\begin{theorem}[Dai \cite{DaAGTS}, Corollary 2.8] $\,$\\ A constrained switching system $S(\graph,\mset)$ is stable if and only if its constrained joint spectral radius satisfies 
\begin{equation*}
\hat \rho(\graph, \mset) < 1.
\end{equation*}

Moreover, for any $\epsilon > 0$, there exists a constant $c > 0$ such that for all trajectories of $S(\graph, \mset)$, 
\begin{equation*}
|x_t| \leq c (\hat \rho(\graph, \mset) + \epsilon)^t|x_0|.
\end{equation*}
\label{thm:stability-cjsr}
\end{theorem}

As a corollary, for any system $S(\graph, \mset)$ and any $\epsilon >0$, 
the system $S(\graph, \mset/(\hat \rho(\graph,  \mset) + \epsilon))$ is  stable. 

%{\color{blue}
%
%Whenever a system admits its own  CJSR as exponential growth rate, it is said to be \textit{non-defective}: 
%
%\begin{definition}
%A system $S( \graph ,  \mset)$ is \textit{non-defective} if there exists a constant $K >  0$ such that, for all $t$, 
%\begin{equation}
%\sup\{\|A_p\| :  p \in \graph, \, |p| = t\} \leq K \hat \rho( \graph, \mset)^t.
%\end{equation}
%
%Alternatively, a system is non-defective if the normalized system $S(\graph, \mset/ \hat \rho(\graph, \mset) )$ has all its admissible products bounded.
%\label{def:non-def}
%\end{definition}
%}

It can be shown \cite{KoTBWF,WaSOLA}, 
 that the  CJSR of $S(\graph(V,E), \mset)$ is equivalent to the JSR of an arbitrary switching system
  in dimension $|V|n$. As a direct consequence, all the tools developed for the study of arbitrary switching systems can be applied for constrained switching systems, with the (potentially very heavy) cost of having to work in higher dimensions.
    
  In what is  developed below, 
  we focus on the system $S(\graph, \mset)$ as it is,
    formulating results without affecting the dimension $n$ of the system.

%%%%%%%%%%%%%%%%%%%%%%%%%%%%%%%%%%%%%%%%%%%%%%%%%%%%%%%%%%%%%%
\section{Multinorms}
\label{sec:multi}
%%%%%%%%%%%%%%%%%%%%%%%%%%%%%%%%%%%%%%%%%%%%%%%%%%%%%%%%%%%%%%

For arbitrary switching systems, one can also define the JSR as follows:
\begin{proposition}[e.g. \cite{JuTJSR}, Proposition 1.4]
$\,$\\
The joint spectral radius of a set of matrices $\mset$ is given by 
\begin{equation}
 \rho(\mset) = \inf_ {\|\cdot\|}\max_{A \in \mset} \{\|A\|\},
 \label{eq:JSRnorm}
\end{equation}
where the infimum is taken over all induced matrix norms in $\reels^{n \times n}.$
\label{prop:JSRnorms}
\end{proposition}

Proposition \ref{prop:JSRnorms} implies that any stable arbitrary switching system  admits a norm as \textit{common Lyapunov function}.

It is straightforward to find examples of constrained switching systems that fail to have this property: this is the case for instance in the example in Section \ref{sec:examp} 
below.  
Hence, it is natural to seek for a generalization of the algebraic notion of norm, 
which would act as a characterization of stability of constrained switching systems.

In this section, we introduce such a generalization under the name of \textit{multinorm}.

\begin{definition}
For a system $S(\graph(V,E),\mset)$, $\mset \subset \reels^{n \times n}$, a $\gamma$-multinorm is a set of $|V|$ norms in
 $\reels^n$ $\{ |{\cdot}|_1, \cdots, |{\cdot}|_{|V|} \}$, such that for any edge $e = (v_i, v_j, \sigma) \in E$, $\forall x \in \reels^n$,
\begin{equation}
|A_\sigma x|_j \leq \gamma |x|_i.
\label{eq:mn}
\end{equation}
\label{def:mn1}
 \end{definition}
 
\begin{definition}
%{\color{blue}
%An \textit{extremal multinorm} is a $\gamma$-multinorm with $\gamma  = \hat \rho(\graph,\mset)$.
%}

A \textit{quadratic multinorm} is a multinorm defined on quadratic norms, $|x|_i = \sqrt{x'Q_ix}$ for $Q_i \succ 0$.\\
A \textit{Lyapunov multinorm} is a $\gamma$-multinorm with $\gamma <  1$.
 \end{definition}

The main result of this section is as follows :

\begin{theorem}[Lyapunov multinorms and stability]
A constrained switching system $S( \graph, \mset)$ is stable if and only if it admits a Lyapunov multinorm.
\label{thm:stab-mn}
\end{theorem}

The result is obtained through the following Lemmas:

 \begin{lemma}
If a system $S(\graph, \mset)$ admits a $\gamma$-multinorm, then $\hat \rho(\graph, \mset) < \gamma$.
 \label{lemma:lower-bound1}
 \end{lemma}

  \begin{lemma}
For any system $S(\graph, \mset)$, for any $\epsilon > 0$, there exists a $\left ( \hat \rho(\graph,\mset) + \epsilon \right ) $-multinorm.
 \label{lemma:mnexists}
 \end{lemma}
 See Annex \ref{annex:lemmas} for the proofs of Lemmas \ref{lemma:lower-bound1} and \ref{lemma:mnexists}.
 
 \begin{proof}[Proof of Theorem \ref{thm:stab-mn}]
When $\hat \rho \geq 1$, there simply cannot be any Lyapunov multinorm (Lemma \ref{lemma:lower-bound1}), and in the other case,
a Lyapunov multinorm  is obtained by taking any 
$( \hat \rho+\epsilon)$-multinorm (Lemma \ref{lemma:mnexists}), 
for $\epsilon < 1 - \hat \rho$ . 
\end{proof}

Concepts similar to Lyapunov multinorms have previously appeared in the literature
 (see \cite{LeDuUSOD, BrMLFA, AhARAH, AhJuPaRoJSRP, DaRiSAAC} for previous works making use
of multiple Lyapunov functions),
but the spirit here is a bit different.
Our goal is to find the suitable generalization of the concept of norm for constrained switching systems.
We do not make any algorithmic consideration in Definition \ref{def:mn1}. 
    In particular, we do not restrict the algebraic nature of the considered norms (i.e., quadratic, polynomial,...). 

The benefit of introducing multinorms is twofold.    
First, they allow to express the CJSR in a manner similar to $(\ref{eq:JSRnorm}).$  
The following result,
 a direct corollary of Lemmas  $\ref{lemma:lower-bound1}$ and $\ref{lemma:mnexists}$,  
 generalizes Proposition \ref{prop:JSRnorms} to constrained switching systems :
\begin{proposition}
For any system $S(\graph, \mset)$,  $\hat \rho (\graph, \mset)$ is the infimum over all $\gamma > 0$ such that there exists a $\gamma$-multinorm for $S(\graph, \mset)$.
\end{proposition}
Second, as shown in the next section, multinorms provide a framework to easily develop converse Lyapunov theorems for constrained switching systems.

%{\color{blue}
%In the arbitrary switching case, a norm $\| \cdot \|^*$ is called \textit{extremal} if
%$$ \rho(\mset) = \max_{A \in \mset} \{\|A\|^*\}.$$ 
%It has been shown that an arbitrary switching system $S(\graph, \mset)$
% admits an extremal norm if and only if it is non-defective (Definition \ref{def:non-def}).
% 
%
%The result can be generalized to constraint switching systems using multinorms:
%
% \begin{corollary}
% A constrained switching system admits an extremal multinorm if and only if it is non-defective.
% \label{cor:nondef} 
% \end{corollary} 
%  The proof is given in Annex \ref{annex:lemmas}.
%}

\section{Converse Lyapunov Theorems}
\label{sec:converse}

In this section, we generalize the Converse Lyapunov Theorems \ref{thm:ando} and \ref{thm:ptasjsr} stated below to constrained switching systems using multinorms.

\begin{theorem}[Ando and Shih \cite{AnShSC}]
$\,$\\
Any \textit{arbitrary switching system} on a set of matrices
 $\mset \subset \reels^{n \times n}$ such that 
 $\hat \rho(\mset) < 1/\sqrt{n}$, admits a \textit{common quadratic Lyapunov function}.
\label{thm:ando}
\end{theorem}

\begin{proposition}[e.g. \cite{JuTJSR}]
Given a set of matrices $\mset$ and an integer $T \geq 1$,  
$$ \hat \rho(\mset)^T = \hat \rho(\mset^T).$$ 
\label{prop:products}
\end{proposition}

Putting together Theorem \ref{thm:ando} and Proposition \ref{prop:products}, one gets the following result:

\begin{theorem}
If an \textit{arbitrary switching system} on a set of matrices
 $\mset \subset \reels^{n \times n}$ is such that 
 $\hat \rho(\mset) < n^{-1/(2T)}$ 
for an integer $T \geq 1$, then the \textit{arbitrary switching system} on the set $\mset^T$ admits a \textit{common quadratic Lyapunov function}.

\label{thm:ptasjsr}
\end{theorem}

\subsection{Quadratic Multinorms}
\label{subsec:quad1}

John's Ellipsoid Theorem \cite{JoEPWI} is a well known result from convex geometry 
that has proven to be of great interest in optimization and 
for approximating the JSR of arbitrary switching systems  in particular \cite{BlNeOTAO}.

\begin{theorem}[John's Ellipsoid Theorem \cite{JoEPWI,BlNeOTAO}]
 Let $K \subset \reels^n$ be a compact convex set with non-empty interior.
  Then there is an ellipsoid $E$ with center $c  \in \reels^n$ such that 
  the inclusions $E \subseteq K \subseteq n(E-c)  + c$ hold. 
  If $K$ is symmetric about the origin $(K = -K)$ the constant $n$ can be changed into $\sqrt{n}$ (and $c = 0$).
 \label{thm:john}
 \end{theorem}

%As it was mentioned in Section \ref{sec:multi}, quadratic multinorms were already in use in, for example, \cite{LeDuUSOD} and \cite{EsDuCOLS} within a framework of path-dependent Lyapunov functions. We will explore this framework later on.
% Their initial theoretical insight is however based on an operator theoretical framework (\cite{DuLaANAF}), whereas ours is based on the material presented in Section \ref{sec:multi}.

The approximation of a multinorm using a quadratic multinorm is not only computationally tractable (see Corollary \ref{cor:ptas1} below) but, 
 together with John's Ellipsoid theorem,
  it also allows to generalize the converse Lyapunov theorem of Ando and Shih (Theorem \ref{thm:ando}) to constrained switching systems.

\begin{theorem}
Any system $S(\graph, \mset)$,  with $$\hat \rho(\graph, \mset) < 1 / \sqrt{n}$$ admits a \textit{quadratic  Lyapunov multinorm}.
\label{thm:cvslya1}
\end{theorem}

\begin{proof}
We claim that, for any system $S(\graph, \mset)$, for any $\epsilon > 0$, there exists $\gamma \in \reels$, 
\begin{equation}
\gamma \leq \sqrt{n} (\hat \rho+\epsilon),
\label{eq:gammabound}
\end{equation}
such that the system admits a \textit{quadratic $\gamma$-multinorm}. 
 Let $\{|{\cdot}|_i, i = 1, \cdots, |V|\}$ be a $\left (\hat \rho + \epsilon \right )$-multinorm
  for the system, as given by Lemma \ref{lemma:mnexists}. 
By John's Ellipsoid Theorem \ref{thm:john}, for each norm in the multinorm, there exists a quadratic form $Q_i \succ 0$ such that
$$ \sqrt{x^\tran Q_i x} \leq |x|_i \leq \sqrt{n} \sqrt{x^\tran Q_ix}, \, \forall x \in \reels^n.$$

From Definition \ref{def:mn1} a $\left (\hat \rho + \epsilon \right )$-multinorm,
we know that for each  edge $(v_i, v_j,\sigma) \in E$, 
$$  |A_\sigma x|_j \leq (\hat \rho + \epsilon) |x|_i, $$
and therefore,

$$ \sqrt{(A_\sigma x)^\tran Q_j(A_\sigma x)}   \leq
 \sqrt{n} (\hat \rho + \epsilon) \sqrt{x^\tran Q_ix}.$$
 
 This shows that there exists always a quadratic $\gamma$-multinorm with $\gamma$ 
 satisfying $(\ref{eq:gammabound})$.

In conclusion, if $\hat \rho < 1/\sqrt{n}$, we can always find a $\gamma < 1$ satisfying (4),
and therefore, the system always admits a \textit{quadratic Lyapunov multinorm}.

\end{proof}

In the next corollary, we formulate a quasi-convex optimization program  
(see \cite{BoCO} for a detailed definition)
that can be solved in order to compute quadratic multinorms.

\begin{corollary}

Consider a system $S(\graph(V,E), \mset)$.\\ The value $\gamma^*$ such that 
\begin{equation}
\begin{aligned}
\gamma^*  = & \inf_{ Q  }  \gamma \\
 \mbox{s.t. } & \forall \, (v_i, v_j,  \sigma) \in E,  \\
 & \qquad {} -A_\sigma^\tran Q_j^{}A_\sigma^{} + \gamma^2 Q_i^{} \succ 0; \\
 & \forall \, i \in \{ 1, \cdots, |V|\}, \, Q_{i} \succ 0.
\end{aligned}
\label{eq:ptas_prog}
\end{equation}
satisfies the following inequalities : 
\begin{equation}
1  \leq \frac{\gamma^*}{ \hat \rho(\graph, \mset)} \leq {\sqrt{n}}.
\label{eq:ptas1}
\end{equation}
\label{cor:ptas1}
\end{corollary}

\begin{proof}
For a fixed $\gamma > 0$, the program described by (\ref{eq:ptas_prog}) has a solution (in terms quadratic forms) 
if and only if  $S(\graph,  \mset/\gamma)$ admits a quadratic Lyapunov multinorm. 

By Lemma \ref{lemma:lower-bound1} and by homogeneity of the CJSR, we have $\hat \rho( \graph, \mset) \leq \gamma$ for any solution of (\ref{eq:ptas_prog}).

Conversely, if the program (\ref{eq:ptas_prog}) does not admit a solution for $\gamma$, by Theorem \ref{thm:cvslya1} and homogeneity of the CJSR,
$\hat \rho (\graph, \mset) \geq \gamma / \sqrt{n}$.
 \end{proof}

In practice, the optimization program described by (\ref{eq:ptas_prog}) 
can be solved using a standard SDP solver together with a bisection algorithm.

%\begin{proof}
%Theorem \ref{thm:app0} guarantees that solving the optimization problem of Corollary \ref{cor:ptas1} will yield a $\gamma$-multinorm for some $\gamma \leq \sqrt{n} \rho(G,M)$. 
%\end{proof}

\subsection{Hierarchy of Converse Lyapunov Theorems}
\label{subsec:quad2}

We will now generalize Proposition \ref{prop:products} and Theorem \ref{thm:ptasjsr} to constrained switching systems.
For a given integer $T \geq 1$, the arbitrary switching system on the set of products $\mset^T$ studied in Theorem \ref{thm:ptasjsr} actually describes 
the dynamics of
$y_s = x_{sT}$, $s = 1, 2, \ldots.$
In the constrained case, these dynamics can be obtained by performing a set of well defined operations on the automaton $\graph$ and  matrix set $\mset:$

 \begin{definition}
 Given a directed labeled graph $\graph(V,E)$, the graph $\graph_T(V,E_T)$ is the graph obtained by associating to each path $p \in \graph$, $|p| = T$, from node $v_i \in V$ to  node $v_j \in V$, 
 a unique edge $e = (v_i,v_j,\{\sigma(1), \ldots, \sigma(T)\}) \in E_T$, with $\sigma(i)$ being the label on the $i$'th edge in $p$.

Given a system $S(\graph,\mset)$,  
the system $S(\graph_T,\mset^T)$ is a system on the graph defined above such
 that to each edge $e = (v_i,v_j,\{\sigma(1), \ldots, \sigma(T)\})\in E_T$  is associated a matrix
 $$ A_e^{} =  A_{\sigma(T)}^{} \cdots A_{\sigma(1)}^{} \in \mset^T.$$
 \label{def:T-lift}
  \end{definition}  

  \begin{proposition}
For any integer $T \geq 1$, the CJSRs of the systems $S(\graph, \mset)$ and $S(\graph_T, \mset^T)$ are such that
  $$ \hat \rho(\graph, \mset)^T = \hat \rho(\graph_T,\mset^T). $$ 
\label{proposition:equal-Tlift}
\end{proposition}

\begin{proof}
Let $ \hat \rho_k(\graph_T, \mset^T)$
 be as in Definition \ref{def:GJSR} for the system $S(\graph_T,\mset^T)$.
 From this definition, we get
$$ \hat \rho_k(\graph_T, \mset^T) = \hat \rho_{kT}(\graph, \mset)^T.  $$

Since the subsequence $\hat \rho_{kT}(\graph, \mset)$
converges to $ \hat \rho(\graph, \mset)$ as $k \rightarrow \infty$,
 by continuity of the exponentiation,
 $\hat \rho_{kT}(\graph, \mset)^T$ converges to $ \hat \rho(\graph, \mset)^T$ as $k \rightarrow \infty$.

\end{proof}

We are now able to state the main theorem of this paper, 
a Converse Lyapunov Theorem generalizing Theorem \ref{thm:ptasjsr}:

\begin{theorem}[Converse Lyapunov Theorem]
$\,$\\
If for an integer $T \geq 1$ the system $S(\graph, \mset)$ is stable with 
$$\hat \rho (\graph, \mset) \leq n ^{-1/(2T)}$$
then the associated system $S(\graph_T,\mset^T)$ (see Definition \ref{def:T-lift}) admits a \textit{quadratic Lyapunov multinorm}.
\label{thm:badass}
\end{theorem}

\begin{proof}
When $\hat \rho (\graph, \mset) \leq n ^{-1/(2T)}$, 
$\hat \rho(\graph_T, \mset^T) = \hat \rho (\graph, \mset)^T \leq 1/\sqrt{n}$ (Proposition \ref{proposition:equal-Tlift}).
 The proof is then direct from Theorem \ref{thm:cvslya1}.
\end{proof}

For any stable system $S(\graph, \mset)$, 
there exists $0 < \epsilon < 1$ such that $$ \hat{\rho} \leq (1-\epsilon).$$
By Theorem \ref{thm:badass}, given this $
\epsilon, $ and for any integer $T$ satisfying
 $$T \geq - \log(n)/\log( (1-\epsilon)^2), $$ 
 the system $S(\graph_T, \mset^T)$ admits a quadratic Lyapunov multinorm.

 Remark that, as mentioned in the introduction, a (quadratic) Lyapunov multinorm for $S(\graph_T, \mset^T)$ defines a (quadratic) multiple Lyapunov function for $S(\graph, \mset)$ that is guaranteed to decrease after every $T$ time steps.

From these results, we obtain a finite time approximation scheme for the CJSR 
using quadratic multinorms.

\begin{theorem}[Approximation scheme]
$\,$\\
For a given system $S(\graph, \mset)$ in dimension $n$, for a given \textit{relative precision} $\epsilon$, 
let $T$ be an integer such that 
$$T \geq {\log(n)}/{\log( (1+\epsilon)^2)}.$$

The solution $\gamma^*$ of the semidefinite program of Corollary 4.6 applied to the system $S(\graph_T, \mset^T)$ is such that

$$ \hat \rho(\graph, \mset) \leq (\gamma^*) ^{1/T} \leq
 (1+\epsilon )\hat \rho(\graph, \mset). $$ 
 \label{thm:approxTot}
\end{theorem}
 
\begin{proof}

Given an integer $T \geq 1$, solving the program of Corollary \ref{cor:ptas1} for $S(\graph_T, \mset^T)$ allows to retrieve
 an estimate $\gamma^*$ of $\hat \rho(\graph_T,\mset^T)$ satisfying $\gamma^* \leq \sqrt{n} \hat \rho(\graph_T,\mset^T).$
 
 Applying  Proposition \ref{proposition:equal-Tlift}, we have 
 
 $$(\gamma^* )^{1/T} \leq n^{1/(2T)} \hat \rho(\graph,\mset).$$
 
 The value of $T$ is then chosen such that $n^{1/(2T)} - 1 < \epsilon$.
 
\end{proof}

%%%%%%%%%%%%%%%%%%%%%%%%%%%%%%%%%%%%%%%%%%%%%%%%%%%%%%%%%%%%%%%%%%%%%%%%%%%%%%%%%%%%%%%%%%%%%%%%%%%%%%%%%%%%%%%%%%%%%%%%%

%%%%%%%%%%%%%%%%%%%%%%%%%%%%%%%%%%%%%%%%%%%%%%%%%%%%%%%%%%%%%%%%%%%%%%%%%%%%%
\subsection{Path-Dependent Lyapunov Functions  }
\label{subsec:pathdep}

In the works of Bliman and Ferrari-Trecate \cite{BlFeSAOD} and Lee and Dullerud \cite{LeDuUSOD}, 
the authors introduced multiple quadratic Lyapunov functions that associate a quadratic form 
per \textit{past trajectory over a finite amount of memory}. 
This amount of memory is an integer parameter defining how complex  the multiple quadratic Lyapunov functions are.
 By gradually increasing this complexity,
such tools provide a hierarchy of sufficient conditions for stability. 
However, they only provide a necessary condition of stability for \textit{asymptotically high complexity} (e.g. \cite{LeDuUSOD}, Theorem 9).

In this section, we provide a converse theorem for the existence of a path-dependent Lyapunov function.
More precisely, given an integer $T \geq 1$, we give a sufficient condition under which a system $S(\graph, \mset)$ 
admits a path-dependent Lyapunov function with memory $(T -  1)$.

 \begin{definition}
Given a system $S(\graph, \mset)$ and an integer $T \geq 0$, 
 a path-dependent Lyapunov function with memory $T$ is a set of quadratic forms such that :
 
 For each \textit{node} $v_i \in V$ and \textit{path} $p \in G$ of length $T$ ending at $v_i$, 
 there is a quadratic form $Q_{v_i; p }$.   
 
 For all paths $ \{ e(0), \ldots, e(T) \} \in \graph$ of length $T+1$, with $e(T) = (v_i, v_j, \sigma)$,
\begin{equation}
\begin{aligned}
&A_{\sigma}^\tran Q^{}_{v_j; \,\{ e(1), \ldots, e(T) \} } A_{\sigma}^{} - \\
& \qquad \qquad \qquad{}  Q^{}_{v_i; \,\{ e(0), \ldots, e(T-1) \} } \prec 0.\\ 
\end{aligned}
\label{eq:pathdep}
\end{equation}

$\,$
\end{definition}
  
Following this definition, a quadratic Lyapunov multinorm for $S(\graph, \mset)$ is a path-dependent Lyapunov function with memory $ 0$.

\begin{theorem}
  $\,$\\
For a  system $S(\graph,\mset)$, if $S(\graph_T, \mset^T)$ admits a quadratic Lyapunov multinorm, 
then $S(\graph, \mset)$ admits a path-dependent Lyapunov function with memory $(T-1)$.
\label{proposition:tliftandpdep}
\end{theorem}
\begin{proof}
The full proof is  given  in Annex \ref{annex:path-dep}.
\end{proof}

As a direct consequence, converse Lyapunov theorems and approximation schemes such as Theorems \ref{thm:badass} and \ref{thm:approxTot} can be formulated
in terms of path-dependent Lyapunov functions.

%%%%%%%%%%%%%%%%%%%%%%%%%%%%%%%%%%%%%%%%%%%%%%%%%%%%%%%%%%%%%%%%%%%%%%%%%%%%%%%%%%%%%%%%%%%%%%%%%%%%%%%%%%%%%%%%%%%%%%%%%%%
\section{Numerical illustration}
\label{sec:examp}

Consider the graph $\graph$ represented in Figure \ref{fig:automat1}, and the set of 4 matrices $\mset$ containing

\begin{equation*}
\begin{aligned}
A_1 & = \begin{pmatrix} 
-0.67  &  0.26 \\
    0.82  & -0.28 \end{pmatrix},& 
A_2 & = \begin{pmatrix} 

    1.13  &  0.13 \\
   -0.82   & -0.28 \end{pmatrix}, \\
A_3 & = \begin{pmatrix} 
0.46  &  0.18 \\
    0.76  & 0.46 \end{pmatrix},& 
A_4 & = \begin{pmatrix}
0.28 &  0.76 \\
    -0.14  &  0.65 \end{pmatrix}.
\end{aligned}
\label{set:non-inv}
\end{equation*}

\begin{figure}[!ht]
\centering
\includegraphics[scale = 0.7]{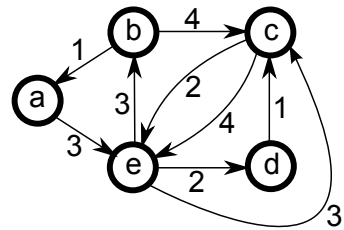}
\caption{The example graph $\graph$. The numbers near the directed edges stand for the labels.}
\label{fig:automat1}
\end{figure}
The matrix $A_2$ is unstable with a spectral radius of $1.15$. 
If the system allowed arbitrary switching, it would be unstable 
because of the presence of $A_2$.

When taking into account the constraints on the switching sequences imposed by $\graph$,
deciding the stability/instability of the system $S(\graph, \mset)$ becomes non-trivial.

The system is however stable. 
To show it, we apply the approximation schemes of Section \ref{sec:converse} 
and retrieve upper bounds on the CJSR of the system that are lower than 1.

In Figure \ref{fig:result_exemple},  for $T = 1, \ldots, 7$,
we compare the results obtained\footnote{Results generated obtained with Matlab on a common laptop (6GB Ram, Intel i7). We used YALMIP to solve the semidefinite programs (http://users.isy.liu.se/johanl/yalmip). Details and codes available on http://perso.uclouvain.be/matthew.philippe .}
 by applying Corollary \ref{cor:ptas1} to the system $S(\graph_T, \mset^T)$
 with those obtained using path-dependent Lyapunov functions with memory {$(T-1)$}.

 \begin{figure}[!ht]
\centering
\includegraphics[width = 0.36\textwidth]{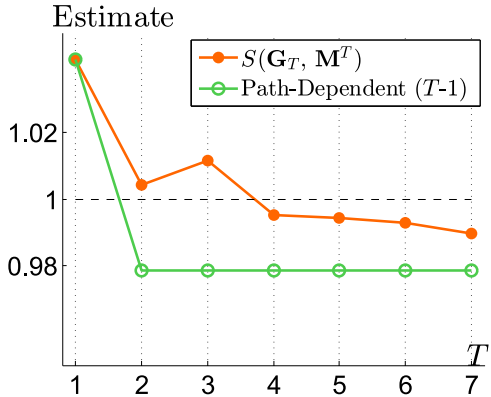}

\caption{ Estimation versus T, the method parameter.}
\label{fig:result_exemple}
\end{figure}

 As expected from Theorem \ref{proposition:tliftandpdep}, path-dependent Lyapunov functions provide better estimates for any fixed $T,$
 showing that the system is stable for $T = 2$ whereas $T = 4$ is needed using $S(\graph_T, \mset^T)$. 
 
 However,  
they also come with a greater computational cost. Path-dependent Lyapunov functions with memory $(T-1)$ have
 one different quadratic form \textit{per path of length $(T-1)$}, (here, for $T = 6$, this amounts to 371 quadratic forms),
   whereas when applying Corollary $\ref{cor:ptas1}$ to
$S(\graph_T, \mset^T)$, the number of quadratic form remains one per node in $\graph$ (here, only 5). 
   As illustrated in Figure \ref{fig:result_time}, for a same level of guaranteed accuracy, obtaining an estimate using $S(\graph_T, \mset^T)$
   is therefore faster. By level of guaranteed accuracy, we refer to the maximum relative error one can get using the above methods with parameter
 $T \geq 1$. From the proof of Theorem \ref{thm:approxTot}, this maximum error is given by $\epsilon(T) = n ^{1/(2T)} - 1$.

 \begin{figure}[!ht]
\centering
\includegraphics[width = 0.45\textwidth]{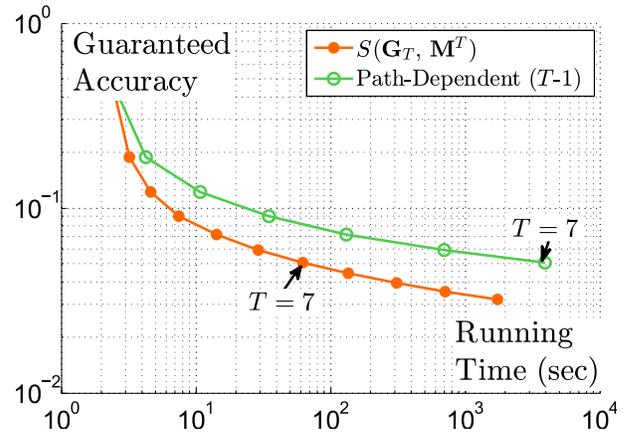}

\caption{ Experimental results. Guaranteed accuracy versus running time (in log-log scale). 
Each mark represents a distinct choice of $T$. The estimation using $S(\graph_{11}, \mset^{11})$  ran twice faster than using 
path-dependent Lyapunov functions with memory $(T-1) = 6$.}
\label{fig:result_time}
\end{figure}

%%%%%%%%%%%%%%%%%%%%%%%%%%%%%%%%%%%%%%%%%%%%%%%%%%%%%%%%%%%%%%%%%%%%%%%%%%%%
\section{Discussion and Conclusion}

In this paper, we provide a stability analysis framework 
for a general class of discrete-time linear switching systems,
 with switching sequences generated by an automaton.
 
 This framework links the asymptotic growth rate, or Constrained Joint Spectral Radius
 of a system with the concept of multinorms.

With these tools in hand, we formulate finite converse Lyapunov theorems 
 for the existence of quadratic Lyapunov multinorms and path-dependent Lyapunov functions.
 We give sufficient conditions on the Constrained Joint Spectral Radius of a system
 under which these Lyapunov functions exists for a given value of their parameter $T$.
 
 These results naturally translate in arbitrary accurate
  approximation schemes for estimating the CJSR of a given system.
 
In further research, we will better compare the different existing 
multiple Lyapunov functions criteria.  We hope that the formalism developed here
 will allow us to conduct a systematic comparison.  Also, we plan to make use of this formalism for optimizing simulation techniques for cyber-physical systems: since multinorms provide a finite algebraic characterization of stable constrained switching system, they could help at the design level, when one builds a constrained switching system in order to simulate a particular cyber-physical system.

\bibliographystyle{IEEEtran}

\newpage
\section{Annex : Detailed proofs}

\subsection{Proofs of Lemmas \ref{lemma:lower-bound1}, \ref{lemma:mnexists}}% \color{blue} and Corollary \ref{cor:nondef} }
\label{annex:lemmas}

 \begin{proof}[Proof of Lemma \ref{lemma:lower-bound1}.]
Consider a path $p \in G$, $|p| = k$, from a node $v_i$ to a node $v_j$.
By the definition of a $\gamma$-multinorm, we have
 \begin{equation}
 |A_px|_j \leq \gamma^k |x|_i.
 \end{equation} 
 
For any norm $|{\cdot}|$, by equivalence of norms in $\reels^n$, there exists $0 < \alpha < \beta$ such that 
 $$ \alpha |x| \leq |x|_i \leq \beta |x|,  $$
 for any norm $|{\cdot}|_i$ in the multinorm and all $x \in \reels^n.$

Therefore, 
\begin{equation}
\begin{aligned}
 \|A_{p}\| & =  \max_{|x| = 1}\frac{|A_{p}x|}{|x|},\\
 & \leq \frac{\beta}{\alpha}\max_{|x| = 1}\frac{|A_{p}x|_j}{|x|_i}, \\
 & \leq \frac{\beta}{\alpha} \gamma^k.
 \end{aligned}
 \label{eq:boundedness}
\end{equation}

Thus, taking the limit $k \rightarrow \infty$, we obtain that the $k$th root of all products is lower than $\gamma$, and by Definition \ref{def:GJSR},  $\hat \rho \leq \gamma$ as well.
 
\end{proof}

\begin{proof}[Proof of Lemma \ref{lemma:mnexists}]
Consider the scaled set of matrices  
$$ \mset = \{A'_i = A_i /  ( \hat \rho + \epsilon ), \, i = 1, \cdots, N \}.$$  

 Given a path $p \in \graph(V,E)$, we use ${A}_p'$ to denote the product of matrices in $\mset'$ associated with the path. 
 Also, let $p \in \graph_i$ denote a path starting at the 
 particular node $v_i \in V$.

For $1 \leq i \leq N$, we define the following norms  
\begin{equation}
|x|_i := \sup_{t \geq 0} \{ |{A}_p' x| : \, p \in \graph_i, \, |p| = t \}, \, i = 1, \cdots, |V|;
\label{eq:nodconst}
\end{equation} 
with $|{\cdot}|$ corresponding to the classical euclidean norm.
We claim that these norms form a well defined $\left (\hat \rho + \epsilon \right)$-multinorm.

 These quantities have all the properties of a vector norm (homogeneity, definiteness, sub-additivity).
 
 They are indeed well-defined (by stability of $S(\graph, \mset)$), sub-additive and homogeneous (due to the use of the euclidean norm $|{\cdot}|$ in their definition), and positive definite (by considering paths of length $t = 0$, we have $|x|_i \geq |x|$).
 
We now show that they are
 $\left ( \hat \rho + \epsilon \right)$-multinorms.
 Rewriting the definition (\ref{eq:nodconst}) and considering an particular edge $e = (v_i, v_j, \sigma)$, we have
 
\begin{equation}
\begin{aligned}
  |x|_{i} & = \sup_{p \in \graph_i} \{ |  A_p'x|_2 \},\\
    & \geq \sup_{q \in \graph_j} \{ |A_{q}' A_\sigma' x|_2 : \, q \in \graph_j\}, \\
    & = | A_\sigma'x|_j.
\end{aligned}
\end{equation}

Since $ A_{\sigma}' = A_{\sigma} / ( \hat \rho + \epsilon)$, we conclude that
\begin{equation}
| A_\sigma x |_{j} \leq (\hat \rho+\epsilon) | x |_{i}, 
\end{equation} 
which proves our claim.
 
\end{proof}

% {\color{blue}
%  
%   \begin{proof}[Proof of Corollary \ref{cor:nondef}]
%
%$\mathbf{\Leftarrow}$ If the system is non-defective, then the proof of Lemma \ref{lemma:mnexists} can be applied verbatim with 
%$\epsilon=0$ (notice that Equation 
%(\ref{eq:nodconst}) still provides a well defined norm since the products $A'_p$ are bounded in norm). 
%
% 
%$\mathbf{\Rightarrow}$ The only-if is given as a direct application of the poof of Lemma \ref{lemma:lower-bound1}. 
%An extremal multinorm implies the non-defectivity of the system with
%$$K = \frac{\beta}{\alpha} $$
%holding as the uniform constant in Definition \ref{def:non-def}.  
%\end{proof}
%
%
%}
\subsection{Proof of Proposition \ref{proposition:tliftandpdep} }
\label{annex:path-dep}

\begin{proof} 
We can assume $T \geq 2$, since for $T = 1$ the claim is trivial.

To ease the reading we will adapt the notations defining a path by referring to it by its succession of nodes and labels, 
$$p = \{v_0,\sigma(1)v_1, \cdots,\sigma(T)v_T\} \in G, $$ 
with $T$ the length of the path. 
Of course, for the notation to be consistent,
 we assume that $$(v_i,v_{i+1},\sigma(i+1)) \in E.$$

For a system $S(\graph_T, \mset^T)$ to admit a quadratic Lyapunov multinorm,
 there must be a constant $\gamma < 1$ such that the following set of LMIs holds true for all paths $p \in G$ of length $T$ (as defined above):
\begin{equation}
 A_p^\tran Q_{v_T}A_p - \gamma^{2T} Q_{v_0}, \preceq 0,
 \label{eq:tliftlmi}
\end{equation}
with $Q_i \succ 0, i = 1, \ldots, |V|$.
We claim that from the solution to the above set of LMIs, one can compute, for each path 
of length $T-1$, a quadratic form $Q_{v_{T}, \{v_1,\sigma(2)v_2, \cdots, \sigma(T) v_{T}\}}$ such that
\begin{align}
&A_{\sigma(T)}^\tran Q^{}_{v_T, \{v_1,\sigma(2)v_2,  \cdots, \sigma(T)v_T\}}A_{\sigma(T)} \nonumber \\
&\qquad   {} - \gamma^{2} Q^{}_{v_{T-1}, \{v_0,\sigma(1)v_1, \cdots, \sigma(T-1)v_{T-1}\}} \preceq 0,
\label{eq:pdep}
\end{align}
which implies the existence of a path-dependent Lyapunov function with memory $T-1$.

Letting $R = Q^{-1}$ denote the inverse of the quadratic form in the corresponding LMIs, using the Schurr complement formula, the LMIs (\ref{eq:tliftlmi}) are equivalent to
\begin{equation}
 A_pR_{v_0}A_p^\tran - \gamma^{2T} R_{v_T}, \preceq 0,
 \label{eq:tliftlmidu}
\end{equation}
and the LMIs (\ref{eq:pdep}) to 
\begin{align}
& A_{\sigma(T)} R^{}_{v_{T-1}, \{v_0,\sigma(1)v_1, \cdots, \sigma(T-1)v_{T-1}\}}A_{\sigma(T)}^\tran \nonumber \\
& \qquad {} - \gamma^{2} R_{v_T, \{v_1, \sigma(2)v_2,  \cdots, \sigma(T)v_T\}} \preceq 0.
\label{eq:pdepdu}
\end{align}

Given a valid solution $R_{v_i}$, $1 \leq i \leq |V|$, to the LMIs \ref{eq:tliftlmidu}, one can check that the following
\begin{align}
& R^{}_{v_{T-1}, \{ v_0,\sigma(1)v_1, \cdots, \sigma(T-1)v_{T-1}\} }  =  \nonumber \\
& \qquad {} R^{}_{v_{T-1}} +  A^{}_{\sigma(T-1)} R^{}_{v_{T-2}} A^\tran_{\sigma(T-1)} / \gamma^2 + \nonumber \\
& \qquad {} ( A^{}_{\sigma(T-1)} A^{}_{\sigma(T-2)} ) R^{}_{v_{T-3}} ( A^{}_{\sigma(T-1)}A^{}_{\sigma(T-2)} )^\tran /\gamma^4 + \cdots  \nonumber \\
& \qquad {} ( A^{}_{\sigma(T-1)} \cdots A^{}_{\sigma(1)} ) R^{}_{v_{0}}  ( A^{}_{\sigma(T-1)} \cdots A^{}_{\sigma(1)} )^\tran /\gamma^{2(T-1)},
\end{align}

satisfy the LMIs (\ref{eq:pdepdu}).
\end{proof}

\end{document}